\documentclass{amsart}
\usepackage{amsmath, centernot}
\usepackage{amssymb}
\usepackage{color}
\usepackage{amscd}
\usepackage[all,cmtip]{xy}
\input xy
\xyoption{all}
\newtheorem{Lemma}{Lemma}[section]
\newtheorem{remark}[Lemma]{Remark}
\newtheorem{remarks}[Lemma]{Remarks}

\newtheorem{lemma}[Lemma]{Lemma}
\newtheorem{proposition}[Lemma]{Proposition}

\newtheorem{definition}[Lemma]{Definition}
\newtheorem{example}[Lemma]{Example}

\newcommand{\Cal}[1]{{\mathcal #1}}

\newcommand{\Hom}{\operatorname{Hom}}
\newcommand{\Triv}{\operatorname{Triv}}

\newcommand{\Mod}{\operatorname{Mod-\!}}

\DeclareMathOperator{\op}{op}

\newcommand{\cmat}{\left(\begin{array}}
\newcommand{\fmat}{\end{array}\right)}

\newcommand{\pre}{\mathbf{Preord}}

\begin{document}
   \title{Pretorsion theories, stable category and preordered sets}
  \author[Alberto Facchini]{Alberto Facchini}
\address{Dipartimento di Matematica ``Tullio Levi-Civita'', Universit\`a di Padova, 35121 Padova, Italy}
 \email{facchini@math.unipd.it}
\thanks{The first author was partially supported by Dipartimento di Matematica ``Tullio Levi-Civita'' of Universit\`a di Padova (Project BIRD163492/16 ``Categorical homological methods in the study of algebraic structures'' and Research program DOR1828909 ``Anelli e categorie di moduli''). The second author was partially supported by GNSAGA of Istituto Nazionale di Alta Matematica and by Dipartimento di Matematica ``Tullio Levi-Civita'' of Universit\`a di Padova (Research program DOR1828909 ``Anelli e categorie di moduli'').}  
 \author[Carmelo Finocchiaro]{Carmelo Finocchiaro}
\address{Dipartimento di Matematica e Informatica, Universit\`a di Catania, 95125 Catania, Italy and Dipartimento di Matematica ``Tullio Levi-Civita'', Universit\`a di Padova, 35121 Padova, Italy}
 \email{carmelo@math.unipd.it}
      \begin{abstract} We show that in the category of preordered sets, there is a natural notion of pretorsion theory, in which the partially ordered sets are the torsion-free objects and the sets endowed with an equivalence relation are the torsion objects. Correspondingly, it is possible to construct a stable category factoring out the objects that are both torsion and torsion-free.\end{abstract}

    \maketitle

\section{Introduction}

Every preorder can be obtained from an equivalence relation and a partial order. In this paper, we show that this can be interpreted in the language of a suitably defined (pre)torsion theory. More precisely, we consider the category $\pre$ of all non-empty preordered sets $(A,\rho)$, and its full subcategories {\bf ParOrd} and {\bf Equiv}, whose objects are all partially ordered sets $(A,\le)$ and all sets $(A,\sim)$, where $\sim$ is an equivalence relation on $A$, respectively. If we view {\bf ParOrd} and {\bf Equiv} as classes of objects, i.e., as subclasses of the class of all objects of $\pre$, we can call {\em trivial objects} of {\bf Preord} the objects of $\pre$ that are both in {\bf ParOrd} and {\bf Equiv}. Trivial objects are the objects of $\pre$ of the form $(A,=)$, where $=$ is the equality relation on $A$. We denote the full subcategory of {\bf Preord}  whose objects are all trivial objects $(A,=)$ by {\bf Triv}. Call {\em trivial morphisms} the morphisms in $\pre$ that factors through a trivial object. Factoring out the category $\pre$ modulo the trivial objects and the trivial morphisms, we get a category \underline{{\bf Preord}}, called the {\em stable category} of {\bf Preord}. The category \underline{{\bf Preord}} turns out to be a pointed category, that is, a category with a zero object. Hence, in \underline{{\bf Preord}}, we have the notions of kernel and cokernel, which have their natural counterpart in $\pre$. We call these natural counterparts prekernels and precokernels. This allows us to define {\em short preexact sequences}: if $f\colon A\to B$ and $g\colon B\to C$ are two morphisms in $\pre$, we say that $\xymatrix{
	A \ar[r]^f &  B \ar[r]^g &  C}$ is a \emph{short preexact sequence} in $\pre$ if $f$ is a prekernel of $g$ and $g$ is a precokernel of $f$. For every object $(A,\rho)$ in $\pre$, let $\sim$ be an equivalence relation on $A$ contained in $ \rho$. Then $$\xymatrix{
	(A,\sim) \ar[r]^k &  (A,\rho) \ar[r]^{\pi} &  (A/\!\sim,\rho)}$$ is a short preexact sequence in $\pre$. Here $k$ is the identity mapping and $\pi$ is the canonical projection.
In particular, for example, let $(A,\rho)$ be any preordered set and  $\simeq_\rho$ denote the equivalence relation on $A$ defined, for every $a,b\in A$, by $a\simeq_\rho b$ if $a\rho b $ and $b\rho a$. Then 
$$\xymatrix{
	(A,\simeq_\rho) \ar[r]^k &  (A,\rho) \ar[r]^\pi &  (A/\!\simeq_\rho,\leq_\rho)}$$
is a short preexact sequence in $\pre$, $(A,\simeq_\rho)\in {}${\bf Equiv} and  $(A/\!\simeq_\rho,\leq_\rho)\in{}${\bf ParOrd}. It is now easy to generalize what we have seen above and define  pretorsion theories for arbitrary categories $\Cal C$. Our definition is as follows. Let $\Cal C$ be an arbitrary category. A {\em pretorsion theory} $(\Cal T,\Cal F)$ for $\Cal C$ consists of two non-empty classes $\Cal T,\Cal F$ of objects of $\Cal C$, closed under isomorphism, satisfying the following two conditions. Set $\Cal Z:=\Cal T\cap\Cal F$, let $\Triv_{\Cal Z}(X,Y)$ be the set of  all morphisms $X\to Y$ in $\Cal C$ that factors though an object of $\Cal Z$, and call {\em $\Cal Z$-preexact} any sequence preexact relatively to $\Cal Z$. The two conditions necessary to define our pretorsion theories 
  $(\Cal T,\Cal F)$ are:
  
  (1) For every object $B$ of $\Cal C$ there is a short $\Cal Z$-preexact sequence $$\xymatrix{
	A \ar[r]^f &  B \ar[r]^g &  C}$$ with $A\in\Cal T$ and $C\in\Cal F$.
	
	(2) $\Hom_{\Cal C}(T,F)=\Triv_{\Cal Z}(T, F)$  for every object $T\in\Cal T$, $F\in\Cal F$. 
	
	\noindent The pair $(${\bf Equiv}$, ${\bf PreOrd}$)$ turns out to be a pretorsion theory in the category $\pre$.

\section{Preorders, partial orders and equivalence relations}
A {\em preorder} on a set $A$ is a relation on $A$ that is reflexive and transitive. We will denote by {\bf Preord} the category of all non-empty preordered sets. Its objects are the pairs $(A,\rho)$, where $A$ is a non-empty set and $\rho$ is a preorder on $A$. The morphisms $f\colon (A,\rho)\to (A',\rho')$ in {\bf Preord} are the mappings $f$ of $A$ into $A'$ such that $a\rho b$ implies $f(a)\rho' f(b)$ for all $a,b\in A$. As usual, when there is no danger of confusion or when the preorder is clear from the context, we will denote the preordered set $(A,\rho)$ simply by $A$.

\begin{remark} {\rm We fix some notations useful in the sequel. 
The category {\bf Preord} has arbitrary coproducts and arbitrary products. Given any family $\{\,(A_i,\rho_i)\mid i\in I\,\}$ of preordered sets,  then its disjoint union $A:=\coprod_{i\in I}A_i (:=\{(a,i)\mid i\in I,a\in A_i \})$ is naturally endowed with the \emph{coproduct preorder} $\rho$  defined, for every $(a,i), (b,j)\in A$, by	
$
(a,i)\rho (b,j)$ if $ i=j$ and $a\rho_i b.
$
Clearly, $(A,\rho)$ is the coproduct of the family. The product in {\bf Preord} of the family is its cartesian product with the componentwise preorder. }
\end{remark}

The main examples of preordered sets $(A,\rho)$ are those in which the preorder $\rho$ is a partial order (i.e., $\rho$ is antisymmetric) or an equivalence relation (i.e., $\rho$ is symmetric). The full subcategories of {\bf Preord} whose objects are all preordered sets $(A,\rho)$ with $\rho$ a partial order (an equivalence relation) will be denoted by {\bf ParOrd} ({\bf Equiv}, respectively).

In our first proposition, we describe preorders on a set $A$. They are obtained from an equivalence relation $\sim$ on $A$ and a partial order on the quotient set $A/\!\!\sim$.

\begin{proposition}\label{preorder} Let $A$ be a set. There is a one-to-one correspondence between the set of all preorders $\rho$ on $A$ and the set of all pairs $(\sim,\le)$, where $\sim$ is an equivalence relation on $A$ and $\le$ is a partial order on the quotient set $A/\!\!\sim$. The correspondence associates to every preorder $\rho$ on $A$ the pair $(\simeq_\rho,\le_\rho)$, where $\simeq_\rho$ is the equivalence relation defined, for every $a,b\in A$, by $a\simeq_\rho b$ if $a\rho b$ and $b\rho a$, and $\le_\rho$ is the partial order on $A/\!\!\simeq_\rho$ defined, for every $a,b\in A$, by $[a]_{\simeq_\rho}\le[ b]_{\simeq_\rho}$ if $a\rho b$. Conversely, for any pair $(\sim,\le)$ with $\sim$ an equivalence relation on $A$ and $\le$ a partial order on  $A/\!\!\sim$, the corresponding preorder $\rho_{(\sim,\le)}$ on $A$ is defined, for every $a,b\in A$, by $a\rho_{(\sim,\le)} b$ if $[a]_{\sim}\le[b]_{\sim}$.\end{proposition}

The objects of  {\bf Preord} that are objects in both the full subcategories    {\bf ParOrd} and {\bf Equiv} are the objects of the form $(A,=)$, where $=$ denotes the equality relation on $A$. We will call them the {\em trivial objects} of {\bf Preord}. The full subcategory of {\bf Preord}  whose objects are all trivial objects $(A,=)$ will be denoted by {\bf Triv}. It is a category isomorphic to the category of all non-empty sets.

A morphism $f\colon (A,\rho)\to (A',\rho')$ in {\bf Preord} is {\em trival} if it factors through a trivial object, that is, if there exist a trivial object $(B,=)$ and morphisms $g\colon (A,\rho)\to (B,=)$ and $h\colon (B,=)\to (A',\rho')$ in {\bf Preord} with $f=hg$.

\begin{Lemma} Let $(A,\rho)$ and $(A',\rho')$ be objects in the category {\bf Preord} and $f\colon A\to A'$ be a morphism. Then $f\colon (A,\rho)\to (A',\rho')$ is a trivial morphism in {\bf Preord} if and only if $a\rho b$ implies $f(a)=f(b)$ for all $a,b\in A$. \end{Lemma}

\begin{proof}Suppose that $a\rho b$ implies $f(a)=f(b)$ for all $a,b\in A$ and consider the canonical factorization $$
\xymatrix{
(A,\rho)\ar[r]^{\pi\ \ }&(A/\!\!\sim_f,=)\ar[r]^{\ \ \overline{f}}&(A',\rho')}
$$ of $f$, where $\sim_f$ is the equivalence relation on $A$ defined, for all $a,b\in A$, by $a\sim_f b$ if $f(a)=f(b)$, $\pi\colon A\to A/\!\!\sim_f$ is the canonical projection and $\overline{f}\colon  A/\!\!\sim_f\to A'$ is the injective mapping induced by $f$. Then
$f=\overline{f}\pi$, and $\pi,\overline{f}$ are morphism  in the category {\bf Preord}.

The converse is trivial. \end{proof}

We will denote by $\Hom(A,A')$ the set of all morphisms $A\to A'$ in the category {\bf Preord}, and by $\Triv(A,A')$ the set of all trivial morphisms $A\to A'$. 

\smallskip

In order to justify the terminology that follows, it is now convenient to introduce some notions of Topology. Recall that a topological space is \emph{Alexandroff-discrete} provided that the intersection of any family of open subsets is an open subset. The full subcategory of {\bf Top} whose objects are all Alexandroff-discrete spaces is isomorphic to the category {\bf Preord}. The category isomorphism associates, to any preordered set $(X,\rho)$, the topological space $(X,\tau_\rho)$, whose open sets are the subsets $A$ of $X$ such that, if $a\in A$, $x\in X$ and $x\rho a$, then $x\in A$. It is immediately seen that $(X,\tau_\rho)$ is an Alexandroff-discrete space. Moreover, for any $x,y\in X$, $x\rho y$ if and only if $y$ belongs to the closure $\overline{\{x\}}$ of $x$ with respect to the topology $\tau_\rho$. A mapping $X\to Y$ is a morphism $(X,\rho)\to (Y,\sigma)$ in {\bf Preord} if and only if it is a continuous mapping $(X,\tau_\rho)\to (Y,\tau_\sigma)$. Conversely, let $(X,\tau)$ be any Alexandrov-discrete space. The associated preorder  $\leq_{\tau}$ on $X$ is defined, for every, $x,y\in X$, by $x\leq_{\tau}y$ if $\overline{\{y\}}\subseteq \overline{\{x\}}$. 

Now we define a suitable quotient category of the category {\bf Preord}, which we will call the {\em stable category} and denote by \underline{{\bf Preord}}. As far as quotient categories are concerned, we follow \cite[pp.~51--52]{ML}. For every object $(A,\rho)$ of {\bf Preord}, we will say that a subset $B$ of $A$ is a {\em clopen} subset if, for every $a\in A\setminus B$ and $b\in B$, $a\centernot{\rho\,} b$ and $b\centernot{\rho\,}  a$.

\begin{remarks}\label{clopen-remark}{\rm 
(1) Let $(A,\rho)$ be a preordered set and let $B\subseteq A$. By definition, $B$ is a clopen subset, according to the definition just given, if and only if $B$ is a clopen subset in the Alexandroff-discrete space $(A,\tau_\rho)$.

(2) Let $(A_i,\rho_i),\ i \in I $, be preordered sets and let $(\coprod_{i\in I}A_i, \rho)$ be their coproduct in $\pre$, where $\rho$ is the coproduct preorder. Fix an index $i\in I$. By definition, if $j\in I\setminus\{i\}$, $a\in A_i$ and $b\in A_j$, then $(a,i)\centernot{\rho\,}(b,j)$ and $(b,j)\centernot{\rho\,}(a,i)$. Hence the canonical image $A_i\times \{i\}$ of $A_i$ in the coproduct  is a clopen subset of $\coprod_{i\in I}A_i$. }
\end{remarks}

\begin{definition} {\rm 
Let $A$ be a preordered set. We say that $A$ is \emph{indecomposable in $\pre$} if $A$ is not isomorphic to $B\coprod C$ in $\pre$, for any (non-empty) preordered sets $B,C$. (Recall that we have defined $\pre$ as the category whose objects are all non-emty preordered sets.)}
\end{definition}

From Remark~\ref{clopen-remark}(2),
we get that:

\begin{lemma}\label{indecomposable-connected}
The following conditions are equivalent for a preordered set $(A,\rho)$:
\begin{enumerate}
	\item $(A,\rho)$ is indecomposable. 
	\item The Alexandroff-discrete space $(A,\tau_\rho)$ is connected, that is, its clopen subsets are only $\emptyset$ and $A$. 
\end{enumerate}
\end{lemma}

For any preordered set $(A,\rho)$, let $\equiv_\rho$ be the equivalence relation generated by $\rho$, that is, the transitive closure of the relation $\rho\rho^o$ (i.e., $x\equiv_\rho y$ if and only if 
 if there exist elements $x_0=x,x_1,\ldots, x_n=y\in A$ such that, for every $i=1,2,\dots n$, either $x_{i-1}\rho x_{i}$ or $x_{i}\rho x_{i-1}$.) 

\begin{lemma}\label{equivalence-classes}
Let $(A,\rho)$ be a preordered set. 
\begin{enumerate} 
	\item  If $B\subseteq A$ is clopen and $a\in B$, then $[a]_{\equiv_\rho}\subseteq B$. 
		\item For any $a\in A$, the equivalence class $[a]_{\equiv_\rho}$ is clopen and connected.
	\item Every clopen subset of $A$ is a union of equivalence classes  modulo $\equiv_\rho$. 
	\item The connected components of $A$ are exactly the equivalence classes of $A$ modulo $\equiv_\rho$. 
\end{enumerate}
\end{lemma}

\begin{proof} (1) Suppose $x\in [a]_{\equiv_\rho}$. By definition, there are elements $x_0,x_1,\ldots,x_n\in A$ such that $a=x_0$, $x=x_n$ and, for every $i=1,2,3,\dots,n$, either $x_{i-1}\rho x_{i}$ or $x_{i}\rho x_{i-1}$. Since $a\in B$ and either $a\rho x_1$ or $x_1\rho a$, the fact that $B$ is clopen implies that $x_1\in B$. Then $x\in B$ by induction.

(2) Assume $x\in [a]_{\equiv_\rho}$ and $y\in A$. If $x\rho y$ or $y\rho x$, then $x\equiv_\rho y$, since $\equiv_\rho$ is generated by $\rho$. It follows $y\in [a]_{\equiv_\rho}$, and this proves that $[a]_{\equiv_\rho}$ is clopen. Suppose there are non-empty clopen sets $U,V$ of $(A,\tau_\rho)$ such that $[a]_{\equiv_\rho}=U\cup V$. If $u\in U,\ v\in V$, then $u\equiv_\rho v \equiv_\rho a$. By (1), it follows that $[a]_{\equiv_\rho}\subseteq U\cap V$, proving that $[a]_{\equiv_\rho}$ is connected. 

(3) Let $B$ be a clopen subset of $A$. By (1), it is clear that $B=\bigcup_{a\in B}[a]_{\equiv_\rho}$. 

(4) By (2), every equivalence class modulo $\equiv_\rho$ is a connected component of $A$, since it is connected and clopen. Conversely, let $Y$ be a connected component of $A$ and note that, for every $a\in Y$, $[a]_{\equiv_\rho}\cap Y$ is clopen in $Y$ relatively to the subspace topology. Since $Y=\bigcup_{a\in Y}([a]_{\equiv_\rho}\cap Y)$ and every union of clopen subsets is clopen because $(A,\tau_\rho)$ is an Alexandroff-discrete space, the fact that $Y$ is connected implies that there exists an element $a_0\in Y$ such that $Y=[a_0]_{\equiv_\rho}\cap Y$. In other words, $Y\subseteq [a_0]_{\equiv_\rho}$ and, since $Y$ is a connected component and $[a_0]_{\equiv_\rho}$ is connected, we infer that $[a_0]_{\equiv_\rho}=Y$. 
\end{proof}

An immediate consequence of Lemmas \ref{indecomposable-connected} and \ref{equivalence-classes} is that $$\displaystyle A=\dot{\bigcup}_{[a]_{\equiv_\rho}\in A\!/\equiv_\rho}[a]_{\equiv_\rho}$$ is the unique coproduct decomposition of a preordered set $(A,\rho)$ into a disjoint union of indecomposable clopen subsets. 

\smallskip

For every pair of objects $(A,\rho),(A',\rho')$ in {\bf Preord}, let $R_{A,A'}$ be the relation on the set $\Hom(A,A')$ defined, for every $f,g\colon (A,\rho)\to (A',\rho')$, by $fR_{A,A'}g$ if there exists a clopen subset $B$ of $A$ such that $f|_B$ and $g|_B$ are two trivial morphisms and $f|_{A\setminus B}=g|_{A\setminus B}$.  By definition, given arbitrary trivial morphisms $f,g\colon A\to B$, we have that $fR_{A,B}g$. In the next lemma, we make use of the terminology of \cite[page~51]{ML}.

\begin{Lemma} The assignment $(A,A')\mapsto R_{A,A'}$ is a congruence on the category {\bf Preord}.\end{Lemma}

\begin{proof} The relations $R_{A,A'}$ are clearly reflexive and symmetric. They are also transitive, because if $f,g,h\colon (A,\rho)\to (A',\rho')$, $fR_{A,A'}g$ and $gR_{A,A'}h$, then there exist clopen subsets $B,C$ of $A$ such that $f|_B$, $g|_B$, $g|_C$ and $h|_C$ are four trivial morphisms, $f|_{A\setminus B}=g|_{A\setminus B}$ and $g|_{A\setminus C}=h|_{A\setminus C}$. Then $B\cup C$ is a clopen subset of $A$. Let us prove that $f|_{B\cup C}$ is a trivial morphism. Suppose $x,y\in B\cup C$ and $x\rho y$. Then we have two cases: either $x,y\in B$, or $x,y\in C\setminus B$. If $x,y\in B$, then $f(x)=f(y)$ because $f|_B$ is a trivial morphism. If  $x,y\in C\setminus B$, then $f(x)=g(x)$ and $f(y)=g(y)$ because $f|_{A\setminus B}=g|_{A\setminus B}$. As $g|_C$ is a trivial morphism, it follows that $g(x)=g(y)$, hence $f(x)=f(y)$. This proves that $f|_{B\cup C}$ is a trivial morphism. Similarly, $h|_{B\cup C}$ is a trivial morphism. Finally, if $x\in A\setminus(B\cup C)$, then $f(x)=g(x)$ because $f|_{A\setminus B}=g|_{A\setminus B}$, and $g(x)=h(x)$ because $g|_{A\setminus C}=h|_{A\setminus C}$. Thus $f(x)=h(x)$. This proves that $f|_{A\setminus (B\cup C)}=h|_{A\setminus  (B\cup C)}$. Therefore $R_{A,A'}$ is an equivalence relation on $\Hom(A,A')$ for each pair $(A,A')$ of objects of {\bf Preord}.

Finally, suppose that $f,g\colon (A,\rho)\to (A',\rho')$ with $fR_{A,A'}g$, $h\colon (B,\sigma)\to (A,\rho)$ and $\ell\colon(A',\rho')\to(C,\tau)$. Then there exists a clopen subset $X$ of $A$ such that $f|_X$ and $g|_X$ are two trivial morphisms and $f|_{A\setminus X}=g|_{A\setminus X}$. Then $h^{-1}(X)$ is a clopen subset of $(B,\sigma)$. The mapping $\ell f h|_{h^{-1}(X)}$ is a trivial morphism, because if $b,b'\in h^{-1}(X)$ and $b\sigma b'$, then $h(b),h(b')\in X$ and $h(b)\rho h(b')$. Since $f|_X$ is a trivial morphism, we get that $fh(b)=fh(b')$. Therefore $\ell fh(b)=\ell fh(b')$. This proves that $\ell f h|_{h^{-1}(X)}$ is a trivial morphism. Similarly, $\ell g h|_{h^{-1}(X)}$ is a trivial morphism. It remains to prove that $\ell f h|_{B\setminus h^{-1}(X)}=\ell g h|_{B\setminus h^{-1}(X)}$. Suppose $b\in B\setminus h^{-1}(X)$. Then $h(b)\in A\setminus X$, so that $f(h(b))=g(h(b))$ because $f|_{A\setminus X}=g|_{A\setminus X}$. Therefore $\ell(f(h(b)))=\ell(g(h(b)))$.\end{proof}

It is therefore possible to construct the quotient category \underline{{\bf Preord}}${}:={}${\bf Preord}${}/R$ \cite[pp.~51-52]{ML}. We will call it the {\em stable category}. Its objects are all non-empty preordered sets $(A,\rho)$, like in {\bf Preord}. The morphisms $(A,\rho)\to (A',\rho')$ are the equivalence classes of $\Hom(A,A')$ modulo $R_{A,A'}$, that is, $$\Hom_{\mbox{\underline{\bf Preord}}}(A,A'):=\Hom_{\mbox{\bf Preord}}(A,A')/{ R_{A,A'}}.$$

There is a canonical functor $F\colon {}${{\bf Preord}}${}\to{}$\underline{{\bf Preord}}. Our notation will be $F(A)=\underline{A}$ for every object $A$ and $F(f)=\underline{f}$ for every morphism $f\colon A\to A'$. 
 
\begin{remark}{\rm 
The functor $F$ preserves coproducts. In fact, let $(A_i,\rho_i)$, $i\in I$, be preordered sets and let $C:=(\coprod_{i\in I}A_i, \rho)$ be their coproduct in $\pre$, so $C$ is the disjoint union and $\rho$ is the coproduct preorder. For any $i\in I$, let $\varepsilon_i\colon  A_i\to C$ be the canonical embedding (for all $x\in A_i, \varepsilon_i(x):=(x,i)$). Fix any preordered set $Y$ and, for all $i\in I$,  consider morphisms $\underline{f_i}\colon \underline{A_i}\to \underline{Y}$ in $\underline{\pre}$. By the universal property of $C$ in $\pre$, there exists a unique morphism $f\colon C\to Y$ (in $\pre$) such that $f_i=f\circ \varepsilon_i$ for all $i\in I$. In particular, $\underline{f_i}=\underline{f}\circ \underline{\varepsilon_i}$. Suppose there is a morphism $\underline{g}\colon \underline{C}\to \underline{Y}$ in $\underline{\pre}$ such that $\underline{f_i}=\underline{g}\circ \underline{\varepsilon_i}$, for any $i\in I$. According to the last condition, for every $i\in I$ there exists a clopen subset $B_i$ of $A_i$ such that $f_i|_{B_i}$ and $(g\circ\varepsilon_i)|_{B_i}$ are trivial morphisms and $f_i|_{A_i\setminus B_i} =(g\circ \varepsilon_i)|_{A_i\setminus B_i}$. By Remark \ref{clopen-remark}(2), the canonical image $B_i\times \{i\}$ of $B_i$ in $C$ is clopen in $C$ and, since $(C,\tau_\rho)$ is Alexandroff-discrete, $B:=\bigcup_{i\in I}B_i\times\{i\}$ is clopen in $C$. Now it is straightforward to check that $f|_B, g|_B$ are trivial morphisms and that $f|_{C\setminus B}=g|_{C\setminus
 B}$. Hence $\underline{ f}=\underline{g}$. This concludes the proof. }
\end{remark}

\smallskip

The category \underline{{\bf Preord}} is a pointed category, that is, has a zero object. In fact, for every object $(A,\rho)$ in {\bf Preord} and every trivial object $(A',\rho')$, all morphisms $(A,\rho)\to(A',\rho')$ and $(A',\rho')\to(A,\rho)$ are trivial, so that all pairs of morphisms $f,g\colon (A,\rho)\to(A',\rho')$ and $f',g'\colon(A',\rho')\to(A,\rho)$ are equivalent modulo $R_{A,A'}$ or $R_{A',A}$. Thus there is a unique morphism $(A,\rho)\to(A',\rho')$ and $(A',\rho')\to(A,\rho)$ in the category \underline{{\bf Preord}}. Notice that the objects of {\bf Preord} are non-empty sets, so that there is always at least one morphism $(A,\rho)\to(A',\rho')$. Hence, all trivial objects of {\bf Preord} become zero objects in the category \underline{{\bf Preord}}. In Proposition~\ref{next}, we will generalize this fact.

\smallskip

For every preordered set $(A,\rho)$, let $A^*$ be the clopen subset of $A$ defined by $$A^*:= \dot{\bigcup}_{[a]_{\equiv_\rho}\in T}[a]_{\equiv_\rho},$$ where $T:=\{\,[a]_{\equiv_\rho}\in 
A/\!\!\equiv_\rho{}\mid |[a]_{\equiv_\rho}|> 1\,\}$. We will say that $(A,\rho)$ is {\em minimal} if either $A$ is not trivial and $A=A^*$, or $A$ is trivial and $|A|=1$. That is, a non-trivial preordered set $A$ is minimal if and only if, for every $a\in A$, there exists $b\in A$, $b\ne a$, such that either $a\rho b$ or $b\rho a$.

\begin{proposition}\label{next} Two non-empty preordered sets $(A,\rho),(A',\rho')$ are isomorphic objects in \underline{{\bf Preord}}  if and only if there exist clopen subsets $X\subseteq A$ and $X'\subseteq A'$ such that $X$ and $X'$ with the induced preorders are trivial objects or empty sets, and $A\setminus X$, $A'\setminus X'$ with the induced preorders are isomorphic objects in {\bf Preord} or $A\setminus X=A'\setminus X'=\emptyset$.\end{proposition}

\begin{proof} {\em Step 1. For every preordered set $(A,\rho)$, the sets $A$ and $A^*$ with the induced preorder are isomorphic objects in} {\underline{\bf Preord}  }.

Fix an element $a_0$ in $A^*$. Let $f\colon A\to A^*$ be defined by $f(x)=x$ for every $x\in A^*$ and $f(x)=a_0$ for every $x\in A\setminus A^*$. Let $\varepsilon\colon A^*\to A$ be the inclusion. Then $f\varepsilon$ is the identity $\iota_{A^*}$ on $A^*$, and $\varepsilon f\,R_{A,A}\iota_{A}$, because (1) $\varepsilon f|_{A^*}=\iota_{A}|_{A^*}$, (2) $\varepsilon f|_{A^*}$ is the constant mapping equal to $a_0$, hence is a trivial morphism, and (3) $\iota_{A}|_{A\setminus A^*}$ is the trivial morphism because ${A\setminus A^*}$ is a trivial object. Hence $A$ and $A^*$ are isomorphic in { $\underline{\pre}$}.

\smallskip

{\em Step 2. If $(A^*,\rho)$ is a minimal preordered set and $f$ is an endomorphism of $(A^*,\rho)$ such that $f\,R_{A^*,A^*}\iota_{A^*}$, then $f=\iota_{A^*}$.}

Suppose $f\,R_{A^*,A^*}\iota_{A^*}$. Then there exists a clopen subset $B$ of $A^*$ such that $f|_B$ and $\iota_{A^*}|_B$ are two trivial morphisms and $f|_{A^*\setminus B}=\iota_{A^*}|_{A^*\setminus B}$. But $\iota_{A^*}|_B$ trivial morphism and $A^*$ minimal implies $B=\emptyset$. Thus $f=\iota_{A^*}$ follows from $f|_{A^*\setminus B}=\iota_{A^*}|_{A^*\setminus B}$. 

\smallskip

{\em Step 3. Two minimal preordered sets $(A^*,\rho),({A'}^*,\rho')$ are isomorphic objects in \underline{{\bf Preord}}  if and only 
they  are isomorphic objects in} {\bf Preord}.

Assume $(A^*,\rho)$ and $({A'}^*,\rho')$ isomorphic in \underline{{\bf Preord}}. Then there exist morphisms $f\colon A^*\to {A'}^*$ and $g\colon {A'}^*\to A^*$ in {\bf Preord} such that $gf\,R_{A^*,A^*}\iota_{A^*}$ and $fg\,R_{{A'}^*,{A'}^*}\iota_{{A'}^*}$.  By Step 2, we have $gf=\iota_{A^*}$ and $fg=\iota_{{A'}^*}$.  Thus $(A^*,\rho)$ and $({A'}^*,\rho')$ are isomorphic in {{\bf Preord}}. The converse is clear.

\smallskip

{\em Step 4. If
$(A,\rho),(A',\rho')$ are isomorphic in} \underline{{\bf Preord}}{\em,  then there exist clopen subsets $X\subseteq A$ and $X'\subseteq A'$ such that $X$ and $X'$ with the induced preorders are trivial objects or empty sets, and $A\setminus X$, $A'\setminus X'$ with the induced preorders are isomorphic objects in {\bf Preord} or $A\setminus X=A'\setminus X'=\emptyset$.}

Suppose $(A,\rho)\cong(A',\rho')$ in \underline{{\bf Preord}}. Set $X:=A\setminus A^*$ and $X':=A'\setminus {A'}^*$ . Then $X$ and $X'$ are clopen subsets of $A$ and $A'$ respectively and, with the induced preorders, they are trivial objects or empty sets. By Step 1, $(A^*,\rho)\cong({A'}^*,\rho')$ in \underline{{\bf Preord}}. By Step 3, $(A^*,\rho)\cong({A'}^*,\rho')$ in {\bf Preord}. 

\smallskip

{\em Step 5. If there exist clopen subsets $X\subseteq A$ and $X'\subseteq A'$ such that $X$ and $X'$ with the induced preorders are trivial objects and $A\setminus X,A'\setminus X'$ with the induced preorders are isomorphic objects in {\bf Preord}, then $(A,\rho)\cong(A',\rho')$ in} \underline{{\bf Preord}}. 

If $X$ and $X'$ are trivial objects, then $A^*=(A\setminus X)^*$ and ${A'}^*=(A'\setminus X')^*$. Thus $A\setminus X\cong A'\setminus X'$ implies $(A\setminus X)^*\cong (A'\setminus X')^*$, hence $A^*\cong {A'}^*$ in {\bf Preord}. Let $f\colon A\to A'$ be any mapping that restrict to an isomorphism $\varphi\colon A^*\to{A'}^*$ and is constant on $A\setminus A^*$. Similarly, let $g\colon A'\to A$ be any mapping that restricts to $\varphi^{-1}\colon {A'}^*\to{}A^*$ and is constant of $A'\setminus {A'}^*$. Then $\underline{f}$ and $\underline{g}$ are mutually inverse isomorphisms in \underline{{\bf Preord}} between $\underline{A}$ and $\underline{A'}$.\end{proof}

\begin{definition}{\rm 
Let $f\colon A\to A'$ be a morphism in {\bf Preord}. We say that a morphism $k\colon X\to A$ in $\pre $ is a \emph{prekernel} of $f$ if the following properties are satisfied: 
\begin{enumerate}
	\item $fk$ is a trivial morphism.
	\item Whenever $\lambda \colon Y\to A$ is a morphism in $\pre$ and $f\lambda$ is trivial, then there exists a unique morphism $\lambda'\colon Y\to X$ in $\pre$ such that $\lambda=k\lambda'$. 
\end{enumerate}}
\end{definition}

Recall that, for every mapping $f\colon A\to A'$, the equivalence relation $\sim_f$ on $A$, associated to $f$, is defined, for every $a,b\in A$, by $a\sim_f b$ if $f(a)=f(b)$. In the next proposition, we show that every morphism in $\pre$ has a prekernel:

\begin{proposition}\label{prekernel} Let $f\colon (A,\rho)\to (A',\rho')$ be a morphism in $\pre$. Then a prekernel of $f$ is the morphism $k\colon (A,\rho{\,\,}\cap\!\sim_f)\to (A,\rho)$, where $k$ the identity mapping and $\sim_f$ is the equivalence relation on $A$ associated to $f$. \end{proposition}

\begin{proof} By definition, we must prove that $fk$ is a trivial morphism and, for every morphism $g\colon (B,\sigma)\to  (A,\rho)$  with $fg$ a trivial morphism, there exists a unique morphism $g'\colon (B,\sigma)\to  (A,\rho{\,\,}\cap\!\sim_f)$ such that $g=kg'$. The uniqueness follows immediately from the fact that $k$ is the identity mapping on the set $A$, so that $g=kg'$ implies that $g=g'$ as mappings of $B$ into $A$. As far as the existence is concerned, we must show that the morphism $g\colon (B,\sigma)\to  (A,\rho)$ is also a morphism $(B,\sigma)\to  (A,\rho{\,\,}\cap\!\sim_f)$, that is, that $b,b'\in B$ and $b\sigma b'$ implies $g(b)\sim_fg(b')$. Now $fg$ is a trivial morphism, so that $b,b'\in B$ and $b\sigma b'$ imply $fg(b)=fg(b')$, hence $g(b)\sim_fg(b')$.\end{proof}

We don't give a proof of the next proposition here, because it will be proved in a much greater generality in Proposition~\ref{prekernel-properties'}.

\begin{proposition}\label{prekernel-properties} 
Let $f\colon A\to A'$ be a morphism in $\pre$ and let $\mu\colon X\to A$ be a prekernel of $f$. The following properties hold. 
\begin{enumerate}
	\item $\mu$ is a monomorphism. 
	\item If $\lambda\colon Y\to A$ is any other prekernel of $f$, then there exists a unique isomorphism $\lambda'\colon Y\to X$ such that $\lambda=\mu\lambda'$. 
\end{enumerate}
\end{proposition}

Since the category \underline{{\bf Preord}} is a pointed category, that is, a category with a zero object $\underline{0}=\underline{T}$ for every trivial object $(T,=)$, for any pair of objects $\underline{A},\underline{A'}$ of \underline{{\bf Preord}}, it is possible to define the {\em zero morphism} $\underline{0}_{\underline{A},\underline{A'}}\colon\underline{A}\to\underline{A'}$ as the composite morphism of the unique morphism of $\underline{A}$ into $\underline{0}$ and the unique morphism of $\underline{0}$ into $\underline{A'}$. For instance, for any constant mapping $c\colon A\to A'$, $\underline{c}\colon\underline{A}\to\underline{A'}$  is clearly   $\underline{0}_{\underline{A},\underline{A'}}\colon\underline{A}\to\underline{A'}$. More generally, for a morphism $g\colon A\to A'$, one has that $\underline{g}=\underline{0}_{\underline{A},\underline{A'}}$ if and only if $g$ is a trivial morphism. By definition, the {\em kernel} of any morphism $\underline{f}\colon\underline{A}\to\underline{A'}$ in \underline{{\bf Preord}} is the equalizer of $\underline{f}$ and the zero morphism $\underline{0}_{\underline{A},\underline{A'}}$.

\begin{proposition}\label{2.12} If $f\colon (A,\rho)\to (A',\rho')$ is a morphism in {\bf Preord} and $k\colon$\linebreak $(A,\rho{\,\,}\cap\!\sim_f)\to (A,\rho)$ is a prekernel of $f$, then $\underline{k}$ is a kernel of $\underline{f}$ in the pointed category \underline{{\bf Preord}}.\end{proposition}

\begin{proof} We must show that $\underline{k}$ satisfies the universal property of kernels. Since $fk$ is trivial (see Proposition \ref{prekernel}), the morphism $\underline{fk}$ is the zero morphism. Take any morphism $k'\colon (K',\tau)\to (A,\rho)$ in $\pre$ such that $\underline{ fk'}$ is the zero morphism, i.e., $fk'$ is trivial. By Proposition \ref{prekernel}, there exists a unique morphism $u\colon (K',\tau)\to (A,\rho{\,\,}\cap \sim_f)$ such that $k'=ku$. In particular, $\underline{k'}=\underline{ku}$. For the uniqueness of $\underline{u}$ in $\underline{\pre}$, we have that the identity $A\to A$ is another prekernel by Proposition~\ref{prekernel}, so that $k$ is an isomorphism by Proposition~\ref{prekernel-properties}(2).
\end{proof}

\begin{example}\label{quotient-preordered-set}  (Quotient preordered set) {\rm Let $(A,\rho)$ be a preordered set and $\sim$ an equivalence relation on $A$. Then $\rho$  induces a well defined preorder $\rho'$ on the quotient set $A/\!\!\sim$ (via the position $[a_1]_\sim\,\,\rho'\,\,[a_2]_\sim$ if $a_1\,\rho\, a_2$ for all $a_1,a_2\in A$) if and only if $\sim{\!\!}\subseteq\rho$. In fact, if the relation $\rho'$ on the quotient set $A/\!\!\sim$  is a well defined reflexive relation, then $a_1\sim a_2$ implies $[a_1]_\sim=[a_2]_\sim$, so $[a_1]_\sim\rho'[a_2]_\sim$, i.e., $a_1\,\rho\, a_2$. This proves that  $\sim{\!\!}\subseteq\rho$. Conversely, if $\sim{\!\!}\subseteq\rho$, then $[a_1]_\sim=[a_2]_\sim$ and $[a_3]_\sim=[a_4]_\sim$ imply $a_1\,\rho\,a_3$ if and only if $a_2\,\rho\,a_4$.

If these two equivalent conditions hold, then the canonical projection $\pi\colon A\to A/\!\!\sim$ is a morphism  in {\bf Preord} and its prekernel is the identity $k\colon (A,\sim)\to (A,\rho)$. 

For simplicity of notation, we will indicate by the same symbol $\rho$ the preorder induced by $\rho$ on the quotient set $A/\!\!\sim$, provided $\sim{\!\!}\subseteq\rho$.}\end{example}

 \begin{remark}\label{proj-epi-stable} {\rm 
Let $A$ be a set and let $\sim,\ \rho$ be, respectively, an equivalence relation and a preorder on $A$ such that $\sim{}\subseteq \rho$. As we have seen, the canonical projection $\pi\colon  (A,\rho)\to (A/\!\!\sim,\rho)$ is a morphism in $\pre$. Then:
\begin{enumerate}
	\item If $B\subseteq A$ is a clopen set, then $\pi(B)$ is clopen in $Q:=A/\!\!\sim$. 
	\item $\underline{\pi}$ is an epimorphism in $\underline{\pre}$. 
\end{enumerate}
\begin{proof}
(1) Take elements $p\in \pi(B)$ and $q \in Q\setminus\pi(B)$, thus $p=\pi(b), q=\pi(x)$ for some $b\in B,x\in A\setminus B$. Since $B$ is clopen in $A$, we have $b\centernot{\rho\,} x$ and $x\centernot{\rho\,} b$ and, by definition, $ p\centernot{\rho\,}q$ and $q\centernot{\rho\,} p$. 

(2) Let $g,h\colon (Q,\rho)\to (T,\tau)$ be morphisms in $\pre$ with $\underline{g\pi}=\underline{h\pi}$ in $\underline{\pre}$. By definition, there is a clopen set $B\subseteq A$ such that $(g\pi)|_B, (h\pi)|_B$ are trivial morphisms and $(g\pi)|_{A\setminus B}=(h\pi)|_{A\setminus B}$. By part (1), $B':=\pi(B)$ is clopen in $Q$ and it is easily seen that $\pi(A\setminus B)=Q\setminus B'$. Then $g|_{B'},h|_{B'}$ are trivial morphisms and $g|_{Q\setminus B'}=g|_{Q\setminus B'}$. Hence $\underline{g}=\underline{h}$, and the conclusion follows. 
\end{proof} }
\end{remark}

It is easy to see that:

\begin{proposition} A morphism  $f\colon (A,\rho)\to (A',\rho')$ in {\bf Preord} is a monomorphism in {\bf Preord} if and only if it is an injective mapping, and is an epimorphism if and only if it is a surjective mapping.\end{proposition}

\begin{remark}{\rm Trivial objects are exactly the projective objects in {\bf Preord}. In fact, if $(X,=)$ is a trivial object, $f\colon (A,\rho)\to (A',\rho')$ is an epimorphism and\linebreak $g\colon (X,=)\to (A',\rho')$ is a morphism, then by the Axiom of Choice there is a mapping $h\colon X\to A$ such that $g=fh$, which is clearly a morphism $h\colon  (X,=)\to (A,\rho)$. Conversely, if $(X,\rho)$ is projective in {\bf Preord}, consider the identity morphisms $\iota\colon (X,\rho)\to(X,\rho)$ and $\iota'\colon (X,=)\to(X,\rho)$. By the projectivity of $(X,\rho)$, there exists a morphism $\varphi\colon (X,\rho)\to(X,=)$ with $\iota'\varphi=\iota$. Then, necessarily, $\varphi$ is the identity mapping and, since $\varphi$ is a morphism, $\rho$ is the equality relation.

Also, all trivial objects $(X,=)$ are projective generators in {\bf Preord}, because if $f,f'\colon (A,\rho)\to (A',\rho')$ are morphisms and $f\ne g$, then there exists $a\in A$ with $f(a)\ne g(a)$. If $c_a\colon (X,=)\to (A,\rho)$ is the constant morphism equal to $a$, then $fc_a\ne gc_a$.}\end{remark}

\begin{definition}{\rm
Let $f\colon A\to A'$ be a morphism in $\pre$. A \emph{precokernel} of $f$ is a morphism $p\colon A'\to X$ such that:
\begin{enumerate}
	\item $pf$ is a trivial map.
	\item Whenever $\lambda\colon A'\to Y$ is a morphism such that $\lambda f$ is trivial, then there exists a unique morphism $\lambda_1\colon X\to Y$ with $\lambda=\lambda_1 p$.
\end{enumerate}}
\end{definition}

If $f\colon (A,\rho)\to (A',\rho')$ is a morphism in {\bf Preord}, consider the canonical projection $c\colon (A',\rho')\to (A'/\zeta_f,\rho'\vee\zeta_f)$, where $\zeta_f$ is the equivalence relation on $A'$ generated by the set $\{\,(f(a_1),f(a_2))\mid a_1,a_2\in A,\ a_1\rho a_2\,\}$ and $\vee$ is the least upper bound in the complete lattice of all preorders on the set $A'$ (that is, $\rho'\vee \zeta_f$ is the preorder on $A'$ generated by $\rho'$ and $\zeta_f$). Thus $\zeta_f$ is the transitive closure of the relation $\{\,(a',a')\mid a'\in A'\,\}\cup\{\,(f(a_1),f(a_2))\mid a_1,a_2\in A,\ a_1\sim_{\rho} a_2\,\}$. Since the equivalence relation $\zeta_f$ is contained in the preorder $\rho'\vee \zeta_f$, it is possible to construct the quotient preordered set (Example \ref{quotient-preordered-set}), and thus $\rho'\vee \zeta_f$ induces a preorder on the quotient set $A'/\zeta_f$ (which will be still denoted by $\rho'\vee \zeta_f$, with a small abuse of notation).

\begin{proposition}\label{precoker-univ} Let $f\colon (A,\rho)\to (A',\rho')$ be a morphism in $\pre$. Then the canonical projection $c\colon (A',\rho')\to (A'/\zeta_f,\rho'\vee\zeta_f)$  is a precokernel of $f$. \end{proposition}

\begin{proof} We must prove that $cf$ is a trivial morphism and that, for every morphism $g\colon (A',\rho')\to (B,\sigma)$  with $gf$ a trivial morphism, there exists a unique morphism $g'\colon (A'/\zeta_f,\rho'\vee\zeta_f)\to  (B,\sigma)$ such that $g=g'c$. If $x,y\in A$ and $x\rho y$, then $f(x)\zeta_f f(y)$ by the definition of $\zeta_f$. So $cf(x)=cf(y)$. This shows that $cf$ is a trivial morphism. Now let $g\colon (A',\rho')\to (B,\sigma)$ be a morphism in {\bf Preord} with $gf$ a trivial morphism. If $x,y\in A$ and $x\rho y$, we get that $gf(x)=gf(y)$. It follows that $\zeta_f\!\subseteq\,\sim_g$. This shows that the position $[a']_{\zeta_f}\mapsto g(a)$ defines a mapping $g'\colon A'/\zeta_f\to B$, and clearly $g$ factors uniquely through the canonical projection $c$ and the mapping $g'$. The mapping $g'$ is a morphism $(A'/\zeta_f,\rho'\vee\zeta_f)\to  (B,\sigma)$ in {\bf Preord}. To see this, consider elements $[a]_{\zeta_f},[\alpha]_{\zeta_f}$ such  that $a(\zeta_f\vee \rho')\alpha$. By definition, there are elements $x_0:=a,x_1,\ldots,x_n:=\alpha\in A'$ such that, for every $0\leq i\leq n-1$, either $x_i\zeta_f x_{i+1}$ or $x_i\rho' x_{i+1}$.  In the first case, since $\zeta_f\subseteq \sim_g$, it follows that $g'(x_i)=g'(x_{i+1})$. In the second case, since $g$ is a morphism, $g(x_i)\sigma g(x_{i+1})$. Thus, in both cases, we have $g'([x_i]_{\zeta_f})=g(x_i)\sigma g(x_{i+1})=g'([x_{i+1}]_{\zeta_f})$. Hence $g'([a]_{\zeta_f})\sigma g'([\alpha]_{\zeta_f})$, because $\sigma$ is transitive.\end{proof}

It is not difficult to prove the next result directly, but it will also be an immediate consequence of our more general Proposition \ref{precokernel-properties'}.

\begin{proposition}\label{precokernel-properties} 
Let $f\colon A\to A'$ be a morphism in $\pre$. 
\begin{enumerate}
	\item Every precokernel of $f$ is an epimorphism. 
	\item If $p\colon A'\to X$, $q\colon A'\to Y$ are cokernels of $f$, then there exists a unique isomorphism $\varphi\colon X\to Y$ such that $q=\varphi p$. 
\end{enumerate}
\end{proposition}

\begin{proposition}\label{coker-epi}
Let $f\colon (A,\rho)\to (B,\sigma)$ be a morphism in $\pre$ and let $\pi\colon (B,\sigma)\to (B/\zeta_f,\sigma\vee \zeta_f)$ be its precokernel. Then
\begin{enumerate}
	\item $\zeta_f\subseteq{} \equiv_\sigma$, where $\equiv_\sigma$ denotes the equivalence relation generated by $\sigma$. 
	\item If $C\subseteq B$ is a clopen set, then $\pi(C)\subseteq B/\zeta_f$ is a clopen set. 
	\item $\underline{\pi}$ is an epimorphism in $\underline{\pre}$. 
\end{enumerate}
\end{proposition}

\begin{proof} (1) By definition, $\zeta_f$ is generated by 
	$$
	G:=\{(f(a_1),f(a_2))\mid a_1,a_2\in A, a_1\rho a_2 \}
	$$
and thus $G\subseteq \sigma\subseteq\,\equiv_\sigma$, because $f$ is a morphism in $\pre$. This allows to conclude. 
	
	(2) Set $Q:=B/\zeta_f$ and consider elements $\eta\in \pi(C)$ and $\lambda\in  Q\setminus \pi(C)$. Then $\eta=\pi(c)$ and $\lambda=\pi(b)$ for some $c\in C$, $b\in B\setminus C$. If $\eta (\sigma\vee \zeta_f)\lambda$, by definition $c(\sigma\vee \zeta_f)b$. Thus there are elements $x_0:=c,x_1,\ldots, x_n=b\in B$ such that, for $0\leq i< n$, either $x_i\sigma x_{i+1}$ or $x_i\zeta_f x_{i+1}$. For $i=0$ we have either $c \sigma x_1$ or $c\zeta_f x_1$. In the first case, we infer $x_1 \in C$, since $C$ is clopen and $c\in C$. In the second case, there are elements $y_0:=c,y_1,\ldots, y_m:=x_1\in B$ such that, for $0\leq j< m$, either $y_j\sigma y_{j+1}$ or $y_{j+1}\sigma y_j$. In particular, since $c\sigma y_1$ or $y_1\sigma c$ and $C$ is clopen, it follows that $y_1\in C$. By induction on $j$, we have $y_m=x_1\in C$. This proves that in both cases $c\sigma x_1$ and $c\zeta_f x_1$, we have $x_1\in C$. By induction on $i$, we get $x_n=b\in C$, a contradiction. 
	
	(3) Argue as in Remark \ref{proj-epi-stable}(2), noting that if $C$ is a clopen subset of $B$, then  $\pi(B\setminus C)=Q\setminus \pi(C)$.
\end{proof}

\begin{proposition}\label{coker-stable}
Let $f(A,\rho)\to (B,\sigma)$ be a morphism in $\pre$ and let $$\pi\colon (B,\sigma)\to (B/\zeta_f, \sigma\vee \zeta_f)$$ be a precokernel of $f$. Then $\underline{\pi}$ is a cokernel of $\underline{ f}$ in $\underline{\pre}$. 
\end{proposition}

\begin{proof}
By Proposition \ref{precoker-univ}, $\pi f$ is trivial, that is, $\underline{\pi f}=0$ in $\underline{\pre}$. Now, let $g\colon (B,\sigma)\to (C,\tau)$ be a morphism in $\pre$ such that $\underline{gf}=0$ in $\underline{\pre}$, that is, $gf$ is trivial. By Proposition \ref{precoker-univ} again, there exists a unique morphism $g'\colon (B/\zeta_f, \sigma\vee \zeta_f)\to (C,\tau)$ in $\pre$ such that $g=g'\pi$, so $\underline{g}=\underline{g'\pi}$. The uniqueness of such a $\underline{g'}$ in $\underline{\pre}$ follows from the fact that $\underline{\pi}$ is an epimorphism in $\underline{\pre}$, in view of Propositions \ref{precoker-univ}, \ref{precokernel-properties}(1) and \ref{coker-epi}(3). 
\end{proof}

\begin{definition}\label{short-preexact} {\rm Let $f\colon X\to Y$ and $g\colon Y\to Z$ be morphisms in $\pre$. We say that $\xymatrix{
	X \ar[r]^f &  Y \ar[r]^g &  Z}$ is a \emph{short preexact sequence} in $\pre$ if $f$ is a prekernel of $g$ and $g$ is a precokernel of $f$.}
\end{definition}

\begin{example}
{\rm Let $A$ be any non-empty set, let $\rho$ be a preorder on $A$ and let $\sim$ be an equivalence relation on $A$ such that $\sim{}\subseteq{} \rho$. By Example \ref{quotient-preordered-set}, $\rho$ induces a preorder on the quotient set $A/\!\sim$, which we also denote by $\rho$. Then $$\xymatrix{
	(A,\sim) \ar[r]^k &  (A,\rho) \ar[r]^{\pi\ \ \ \ }  &  (A/\!\sim,\rho)}$$ is a short preexact sequence in $\pre$, where $k$ is the identity map and $\pi$ is the canonical projection. The equalities $\sim_\pi\cap{\,} \rho={\!}\sim \cap{\,} \rho={}\sim$ and Proposition \ref{prekernel} imply that $k$ is a prekernel of $\pi$. Moreover, by definition $\zeta_k={\!}\sim$ and $\rho\vee \zeta_k=\rho{\,}\vee \sim{\!}=\rho$, hence $\pi$ is a precokernel of $k$ by Proposition \ref{precoker-univ}. 

In particular, if $\simeq_\rho$ is the equivalence relation on $A$ defined by $a\simeq_\rho b$ if $a\rho b $ and $b\rho a$ (so that in particular we have $\simeq_\rho\subseteq \rho$) and $\leq_\rho$ is the partial order on $A/\!\simeq_\rho$ induced by $\rho$ (see Proposition \ref{preorder}), then 
$$\xymatrix{
	(A,\simeq_\rho) \ar[r]^k &  (A,\rho) \ar[r]^{\pi\ \ \ \ \ \ }&  (A/\!\simeq_\rho,\leq_\rho)}$$
is a short preexact sequence in $\pre$.}
\end{example}

\begin{example}\label{seq-indotta-da-morfismo}
{\rm Let $f\colon (A,\rho)\to B$ be a morphism in $\pre$. Consider the canonical prekernel $k\colon (A,\rho{\,\,} \cap \sim_f)\to (A,\rho)$ of $f$ (Proposition \ref{prekernel}). Let $\pi\colon(A,\rho)\to (A/\zeta_k, \rho\vee \zeta_k)$ be the canonical precokernel of $k$, according to Proposition \ref{precoker-univ}. Then $$\xymatrix{
	(A,\rho{\,\,} \cap\sim_f) \ar[r]^{\ \ \ k} &  (A,\rho) \ar[r]^{\pi\ \ \ \ \ \ }  &  (A/\zeta_k,\rho\vee \zeta_k)}$$
is a short preexact sequence in $\pre$. To prove it, we only need to show that $k$ is a prekernel of $\pi$. This  follows from the equalities $\rho\, \cap \sim_\pi=\rho\,\cap \zeta_k=\rho\,\cap \sim_f$ and Proposition \ref{prekernel}.   }
\end{example}

In the next proposition, we determine all short preexact sequences in $\pre$, up to isomorphism. 

\begin{proposition}\label{preexact-characterization}
Let $\xymatrix{
	(X,\rho) \ar[r]^f &  (Y,\sigma) \ar[r]^g &  (Z,\tau)}$ be any short preexact sequence in $\pre$. Then there exists a commutative diagram 
\begin{equation}
\xymatrix{
	(X,\rho) \ar[r]^{f} \ar[d]_{\cong}&  (Y,\sigma) \ar[r]^{g} \ar@2{-}[d]&  (Z,\tau) \ar[d]^{\cong}\\
	(Y,\sigma\cap \sim_g) \ar[r]_{\ \ \ \ k} &  (Y,\sigma) \ar[r]_{\pi\ \ \ \ \ \ }  &  (Y/\zeta_k, \zeta_k\vee \sigma)},\label{ccc'}\tag{$*$}
\end{equation}
where $k$ is the identity map, $\pi$ is the canonical projection and $$\xymatrix{(Y,\sigma\,\cap \sim_g) \ar[r]^{\ \ \ \ k} &  (Y,\sigma) \ar[r]^{\pi\ \ \ \ \ \ } &  (Y/\zeta_k, \zeta_k\vee \sigma)}$$ is a short preexact sequence. 
\end{proposition}

\begin{proof}
 By Proposition \ref{prekernel}, $k$ is a prekernel of $g$. Since, by assumption, $f$ is also a prekernel for $g$, we infer from Proposition \ref{prekernel-properties}(2) the existence of an isomorphism $(X,\rho)\to (Y,\sigma\cap \sim_g)$ which makes the square on the left in diagram (\ref{ccc'}) commute. Now $g$ is a precokernel of $f$, so that the existence of an isomorphism $(Z,\tau)\to (Y/\zeta_k,\zeta_k\vee \sigma)$ making the right square of the diagram commute will follow by showing that $\pi$ is a precokernel of $f$ (Proposition \ref{precokernel-properties}(2)). First, we prove that $\pi f$ is trivial. Fix elements $x,y\in X$ with $x\,\rho\, y$. Since $g$ is a precokernel of $f$, $gf$ is trivial and thus $g(f(x))=g(f(y))$, that is, $f(x)\sim_g f(y)$. Moreover $f(x)\sigma f(y)$, because $f$ is a morphism. Since, by definition, $\zeta_k$ is generated by $\{(y_1,y_2)\in Y\times Y\mid y_1\,(\sigma\cap\sim_g)\,y_2 \}$, it follows that $f(x)\,\zeta_k\, f(y)$, that is, $\pi f(x)=\pi f(y)$. This proves that $\pi f$ is trivial. Now, let $\lambda\colon (Y,\sigma)\to (T,\eta)$ be a morphism with $\lambda f$ trivial. We have to show that there exists a unique morphism $\widetilde{\lambda}\colon (Y/\zeta_k,\zeta_k\vee \sigma)\to (T,\eta)$ such that $\lambda=\widetilde{\lambda}\pi$. It is enough to show the existence of such a $\widetilde{\lambda}$, the uniqueness being a trivial consequence of the fact that $\pi$ is surjective. But, since $g$ is a precokernel of $f$ and $\lambda f$ is trivial, there exists a unique morphism $\lambda_1\colon (Z,\tau)\to (T,\eta)$ such that $\lambda=\lambda_1g$. Now $\zeta_k$ is the equivalence relation generated by $\sigma\,\cap \sim_g$ and $\sim_g$ is an equivalence relation, so $\zeta_k\subseteq \sim_g$. 
Thus $\lambda_1$ induces a well defined mapping $\widetilde{\lambda}\colon (Y/\zeta_k,\zeta_k\vee \sigma)\to (T,\eta)$, $[y]_{\zeta_k}\mapsto \lambda_1(g(y))$ and $\lambda_1$ clearly satisfies $\lambda=\widetilde{\lambda}\pi$. We claim that $\lambda_1$ is a morphism. In order to see this, take $[y]_{\zeta_k},[z]_{\zeta_k}\in Y/\zeta_k$ with $y(\sigma\vee \zeta_k) z$, and let $\beta_1:=y,\beta_2,\ldots, \beta_n:=z\in Y$ be such that, for every $1\leq i<n$, either $\beta_i\sigma\beta_{i+1}$ or $\beta_i\zeta_k\beta_{i+1}$. In the first case, we infer $g(\beta_i)\tau g(\beta_{i+1})$, and thus $\widetilde \lambda([\beta_i]_{\zeta_k})\eta \widetilde{\lambda}([\beta_{i+1}]_{\zeta_k})$, because $g,\lambda_1$ are morphisms. In the second case, we have, by definition, $[\beta_i]_{\zeta_k}=[\beta_{i+1}]_{\zeta_k}$. Thus in both cases we get $\widetilde \lambda([\beta_i]_{\zeta_k})\eta \widetilde{\lambda}([\beta_{i+1}]_{\zeta_k})$ for any $1\leq i<n$. Since $\eta$ is transitive, it follows that $\widetilde{\lambda}([y]_{\zeta_k})\eta \widetilde{\lambda}([z]_{\zeta_k})$, proving that $\widetilde{\lambda}$ is a morphism. 

Finally, $\xymatrix{(Y,\sigma\cap \sim_g) \ar[r]^{\ \ \ k} &  (Y,\sigma) \ar[r]^{\pi\ \ \ \ } &  (Y/\zeta_k, \zeta_k\vee \sigma)}$ is a short preexact sequence as we saw in Example~\ref{seq-indotta-da-morfismo}.\end{proof}

\section{Short exact sequences}

In a pointed category, a {\em short exact sequence} is a pair of morphisms $f\colon A\to B$, $g\colon B\to C$ such that $f$ is a kernel of $g$ and $g$ is a cokernel of $f$. As usual, the notation will be $\xymatrix{0 \ar[r]&
 A \ar[r]^f &  B \ar[r]^g &  C\ar[r]&0}$. We will now describe all short exact sequences in the pointed category \underline{{\bf Preord}} up to isomorphism (similarly to the description of any 
short exact sequence in the category $\Mod R$ of right modules over a ring $R$, which is isomorphic to a short exact sequence of the form $\xymatrix{0 \ar[r]&
 A_R \ar[r] &  B_R \ar[r] &  A/B\ar[r]&0}$ for suitable modules $A_R\le B_R$). 

\begin{proposition}\label{exact!} For every short exact sequence $$\xymatrix{0 \ar[r]&
 \underline{A} \ar[r]^{\underline{f}} &  \underline{(B,\rho)} \ar[r]^{\underline{g}} &  \underline{C}\ar[r]&0}$$ in \underline{{\bf Preord}}, there is a commutative diagram \begin{equation}\xymatrix{0 \ar[r]&
 \underline{A} \ar[r]^{\underline{f}} \ar[d]_{\cong}&  \underline{(B,\rho)} \ar[r]^{\underline{g}} \ar@2{-}[d]&  \underline{C}\ar[r]\ar[d]^{\cong}&0\\
 0 \ar[r]&
 \underline{{(B,\rho{\,\,}\cap\!\sim)}} \ar[r]_{\ \ \ \ \underline{k}}&  \underline{(B,\rho)} \ar[r]_{{\underline{\pi}}\ \ \ } &  \underline{(B/\!\sim,(\rho{\,}\vee\!\sim)')}\ar[r]&0},\label{ccc}\tag{$**$}\end{equation} where $\sim$ is an equivalence relation on $B$, $k$ is the identity and $\pi$ is the canonical projection.\end{proposition}

\begin{proof} 
Let $f\colon A\to (B,\rho),\ g\colon(B,\rho)\to C$ be representatives of $\underline{f},\ \underline{g}$. Consider the sequence of morphisms $$\xymatrix{(B,\rho\,\cap \sim_g) \ar[r]^{\ \ \ \ k} &  (B,\rho) \ar[r]^{\pi\ \ \ \ } &  (B/\zeta_k, \zeta_k\vee \rho)}$$ in $\pre$, where $k$ is the identity and $\pi$ is the canonical projection. By Propositions \ref{prekernel} and \ref{precoker-univ}, $k$ is a prekernel of $g$ and $\pi$ is a precokernel of $k$. Moreover the pair of morphisms $k,\pi$ is a short preexact sequence in $\pre$ by Example \ref{seq-indotta-da-morfismo}. By Propositions \ref{2.12} and \ref{coker-stable}, the canonical image 
$$\xymatrix{\underline{(B,\rho{\,\,}\cap \sim_g)} \ar[r]^{\ \ \ \ \underline k} &  \underline{(B,\rho)} \ar[r]^{\underline{\pi}\ \ \ \ } &  \underline{(B/\zeta_k, \zeta_k\vee \rho)}}$$
is an exact sequence in the stable category $\underline{\pre}$. Since $\underline{f}, \underline{k}$ are kernels of $\underline g$, there is a unique isomorphism $\underline{A}\to \underline{(B,\rho \cap \sim_g)}$ in $\underline\pre$ which makes the square on the left of diagram (\ref{ccc}) commute. In order to get an isomorphism making the square on the right commute, it will suffice to show that $\underline{\pi}$ is a cokernel of $\underline f$. Let $\tau$ be the preorder on $A$ and let $a,b\in A$ be such that $a\tau b$. Since $\underline{f},\underline{g}$ form a short exact sequence, we have, by definition, $\underline{g}\underline{f}=0$, i.e., $gf$ is trivial. Thus $gf(a)=gf(b)$ and, since $f$ is a morphism, $f(a)\rho f(b)$, i.e., $f(a)(\rho\, \cap \sim_g)f(b)$. Since, by definition, $\zeta_k$ is the equivalence relation  generated by $\rho\,\cap \sim_g$, it follows $[f(a)]_{\zeta_k}=[f(b)]_{\zeta_k}$, that is, $\pi f(a)=\pi f(b)$, and this proves that $\pi f$ is trivial, i.e., $\underline{\pi f}=0$. Now take any morphism $\underline{\lambda}\colon\underline{B}\to \underline T$ in $\underline \pre$ satisfying $\underline \lambda \underline f=0$. Let $\lambda\colon B\to T$ be a representative of $\underline{ \lambda}$. By assumption $\underline{g}$ is a cokernel of $\underline{f}$, so there exists a unique morphism $\underline{\lambda_1}\colon\underline{C}\to \underline{ T}$ in $\underline{\pre}$ with $\underline{\lambda}=\underline{\lambda_1 g}$. Let $\lambda_1\colon C\to T$ be a representative of $\underline{\lambda_1}$. By definition, $\zeta_k\subseteq \sim_g$ and thus there is a well defined mapping $\widetilde{\lambda}\colon(B/\zeta_k,\zeta_k\vee \rho)\to T$ such that $\widetilde{\lambda}([y]_{\zeta_k})=\lambda_1g(y)$, for every $y\in B$. By the same argument given in the proof of Proposition \ref{preexact-characterization}, it is easily seen that $\widetilde{\lambda}$ is a morphism in $\pre$. We claim that the canonical image $\underline{\widetilde{\lambda}}$ of $\widetilde{\lambda}$ in $\underline{\pre}$ is such that $\underline{\lambda}=\underline{\widetilde{\lambda}}\underline{\pi}$ (and it is clearly the unique one with this property, because $\underline{\pi}$ is an epimorphism by Proposition \ref{coker-epi}(3)). Since $\underline{\lambda}=\underline{\lambda_1}\underline{g}$, there is a clopen subset $\Lambda$ of $(B,\rho)$ such that $\lambda=\lambda_1g$ on $\Lambda$ and both $\lambda$ and $\lambda_1g$ are trivial on $B\setminus\Lambda$. It easily follows that $\lambda=\widetilde{\lambda}\pi$ on $\Lambda$ and that $\lambda, \widetilde{\lambda}\pi$ are trivial on $B\setminus\Lambda$. This concludes the proof. 
\end{proof}

\begin{proposition} An identity morphism $k\colon (A,\sigma)\to(A,\rho)$ is a prekernel of a morphism in {\bf Preord} if and only if $\sigma=\rho{\,\,}\cap \equiv_\sigma$. Moreover, if these equivalent conditions hold, then $k$ is the prekernel of the canonical projection $\pi\colon (A,\rho)\to (A/\!\!\equiv_\sigma, \rho\vee\equiv_\sigma)$.\end{proposition}

\begin{proof} If $k\colon (A,\sigma)\to(A,\rho)$ is a kernel of a morphism $f\colon (A,\rho)\to (A',\rho')$ in {\bf Preord}, then $k$ is the morphism $k\colon (A,\rho{\,\,}\cap\!\sim_f)\to (A,\rho)$, where $\sim_f$ is the equivalence relation on $A$ defined, for every $x,y\in A$, by $x\sim_fy$ if $f(x)=f(y)$ and $k$ is the identity mapping.  Thus $\sigma=\rho{\,\,}\cap\!\sim_f$. We must prove that $\sigma=\rho\,\cap \equiv_\sigma$. Now $\sigma=\rho{\,\,}\cap\!\sim_f$ implies that $\sigma\subseteq \rho$ and that $\sigma$ is contained in the equivalence relation $\sim_f$. Hence the equivalence relation $\equiv_\sigma$ generated by $\sigma$ is contained in $\sim_f$. Therefore $\rho\,\cap \equiv_\sigma{}\subseteq \rho\,\cap \sim_f{}=\sigma$. Conversely, $\sigma\subseteq\rho$ because the identity $k\colon (A,\sigma)\to(A,\rho)$ is a morphism in {\bf Preord}, and $\sigma\subseteq\,\equiv_\sigma$ trivially. This concludes the proof of one of the implications of the first part of the statement. The rest follows trivially considering the canonical projection $\pi\colon (A,\rho)\to (A/\!\!\equiv_\sigma, \rho\vee\equiv_\sigma)$ and noticing that $\sim_\pi\,=\,\equiv_\sigma$.\end{proof}

  \section{Pretorsion theories}
  
  Now we will present a general setting for the results proved in the previous sections. (Pre)torsion theories in general categories are studied in \cite{GJ}, \cite{GJM} and \cite{JT}. These papers are rather technical, and we will only use a tiny part of the results presented there. Hence, in this section, we give an elementary presentation of the results we need. We are grateful to Marino Gran, Marco Grandis and Sandra Mantovani for some useful suggestions. 
  
  \medskip
    
  Fix an arbitrary category $\Cal C$ and a non-empty class $\Cal Z$ of objects of $\Cal C$. For every pair $A,A'$ of objects of $\Cal C$, we will write $\Triv_{\Cal Z}(A, B)$ for the set of  all morphisms in $\Cal C$ that factors though an object of $\Cal Z$. We will call these morphisms $\Cal Z$-trival.

Let $f\colon A\to A'$ be a morphism in $\Cal C$. We say that a morphism $k\colon X\to A$ in $\Cal C $ is a \emph{$\Cal Z$-prekernel} of $f$ if the following properties are satisfied: 
\begin{enumerate}
	\item $fk$ is a $\Cal Z$-trivial morphism.
	\item Whenever $\lambda \colon Y\to A$ is a morphism in $\Cal C$ and $f\lambda$ is $\Cal Z$-trivial, then there exists a unique morphism $\lambda'\colon Y\to X$ in $\Cal C$ such that $\lambda=k\lambda'$. 
\end{enumerate}

\begin{proposition}\label{prekernel-properties'}
Let $f\colon A\to A'$ be a morphism in $\Cal C$ and let $\mu\colon X\to A$ be a $\Cal Z$-prekernel for $f$. The following properties hold. 
\begin{enumerate}
	\item $\mu$ is a monomorphism. 
	\item If $\lambda\colon Y\to A$ is any other $\Cal Z$-prekernel of $f$, then there exists a unique isomorphism $\lambda'\colon Y\to X$ such that $\lambda=\mu\lambda'$. 
\end{enumerate}
\end{proposition}

\begin{proof}
(1) Let $T$ be an object of $\Cal C$ and $g,h\colon T\to X$ be morphisms in $\Cal C$ such that $\mu g=\mu h$. Since $\mu$ is a $\Cal Z$-prekernel of $f$, the morphism $f\mu$ is $\Cal Z$-trivial, so that, a fortiori, the morphism $f\mu g$ is also $\Cal Z$-trivial. By applying part (2) of the definition of $\Cal Z$-prekernel to $\mu g$, it follows that there exists a unique morphism $u\colon T\to X$ such that $\mu u=\mu g=\mu h$. From the uniqueness of $u$, we get that $u=g=h$. 

(2) By part (2) of the definition of $\Cal Z$-prekernel, there is a unique morphism $\lambda'\colon Y\to X$ such that $\lambda=\mu \lambda'$ and a unique morphism $\mu'\colon X\to Y$ such that $\mu=\lambda\mu'$. It follows that $\lambda=\lambda\mu'\lambda'$ and $\mu=\mu\lambda'\mu'$. By part (1) of this proposition, we obtain that $\mu'\lambda'=1_Y, \lambda'\mu'=1_X$. The conclusion is now clear. 
\end{proof}

Dually, a \emph{$\Cal Z$-precokernel} of $f$ is a morphism $p\colon A'\to X$ such that:
\begin{enumerate}
	\item $pf$ is a $\Cal Z$-trivial morphism.
	\item Whenever $\lambda\colon A'\to Y$ is a morphism and $\lambda f$ is $\Cal Z$-trivial, then there exists a unique morphism $\lambda_1\colon X\to Y$ with $\lambda=\lambda_1 p$.
\end{enumerate}

\smallskip

If $\Cal C^{\op}$ is the opposite category of $\Cal C$, the $\Cal Z$-precokernel of a morphism $f\colon A\to A'$ in $\Cal C$ is the $\Cal Z$-prekernel of the morphism $f\colon A'\to A$ in $\Cal C^{\op}$. Hence, from Proposition \ref{prekernel-properties'}, we get:

\begin{proposition}\label{precokernel-properties'}
Let $f\colon A\to A'$ be a morphism in a category $\mathcal C$ and let $p\colon A'\to X$ be a $\mathcal Z$-precokernel of $f$. Then the following properties hold. 
\begin{enumerate}
	\item $p$ is an epimorphism.
	\item If $q\colon A'\to Y$ is another $\mathcal Z$-precokernel of $f$, then there exists a unique isomorphism $\varphi\colon X\to Y$ satisfying $q=\varphi p$. 
\end{enumerate}
\end{proposition}

Let $f\colon A\to B$ and $g\colon B\to C$ be morphisms in $\Cal C$. We say that $$\xymatrix{
	A \ar[r]^f &  B \ar[r]^g &  C}$$ is a \emph{short $\Cal Z$-preexact sequence} in $\Cal C$ if $f$ is a $\Cal Z$-prekernel of $g$ and $g$ is a $\Cal Z$-precokernel of $f$.
	
Clearly, if $\xymatrix{
	A \ar[r]^f &  B \ar[r]^g &  C}$ is a short $\Cal Z$-preexact sequence in $\Cal C$, then\linebreak $\xymatrix{
	C \ar[r]^g &  B \ar[r]^f &  A}$ is a short $\Cal Z$-preexact sequence in $\Cal C^{\op}$.

\smallskip

It is now clear that if $\mathcal C:=\pre$ and $\mathcal Z$ is the class of all objects of type $(A,=)$ (where $A$ is any non-empty set and $=$ is the equality relation on $A$), then the $\mathcal Z$-prekernel and the $\mathcal Z$-precokernel of any morphism $f$ in $\pre$ exist and coincide with the prekernel (resp., cokernel) of $f$. In particular, a short $\mathcal Z$-preexact sequence in $\pre$ is a short preexact sequence, as in Definition \ref{short-preexact}.
  
  \begin{definition} {\rm Let $\Cal C$ be an arbitrary category. A {\em pretorsion theory} $(\Cal T,\Cal F)$ for $\Cal C$ consists of two non-empty classes $\Cal T,\Cal F$ of objects of $\Cal C$, closed under isomorphism, satisfying the following two conditions. Set $\Cal Z:=\Cal T\cap\Cal F$. 
  
  (1) For every object $B$ of $\Cal C$ there is a short $\Cal Z$-preexact sequence $$\xymatrix{
	A \ar[r]^f &  B \ar[r]^g &  C}$$ with $A\in\Cal T$ and $C\in\Cal F$.
	
	(2) $\Hom_{\Cal C}(T,F)=\Triv_{\Cal Z}(T, F)$  for every object $T\in\Cal T$, $F\in\Cal F$.}\end{definition}
    
  \begin{Lemma} Let $\Cal C$ be a category, $\Cal Z$ a non-empty class of objects of $\Cal C$ and $$\xymatrix{
	A \ar[r]^f &  B \ar[r]^g &  C}$$ a short $\Cal Z$-preexact sequence. Then:
	
	{\rm (a)} $f$ is $\Cal Z$-trivial if and only if $g$ is an isomorphism.
	
	{\rm (b)} $g$ is $\Cal Z$-trivial if and only if $f$ is an isomorphism.\end{Lemma}
	
	\begin{proof} (a) Suppose $f$ $\Cal Z$-trivial. By definition, the identity morphism $1_B\colon B\to B$ is clearly a $\Cal Z$-precokernel of $f$. By the uniquess up to isomorphism of $\Cal Z$-precokernels (see Proposition \ref{precokernel-properties'}), there is a unique isomorphism $h\colon C\to B$ such that $hg=1_B$. Hence $g$ is an isomorphism.
	
	Conversely, assume that $g$ is an isomorphism. Since $gf$ is $\Cal Z$-trivial, i.e., $gf$ factors though an object in $\Cal Z$, the same holds for $f$, that is, $f$ is $\Cal Z$-trivial.
	
	(b) follows from (a) passing to the opposite category $\Cal C^{\op}$.\end{proof}
	
	\begin{proposition} Let $(\Cal T,\Cal F)$ be a pretorsion theory in a category $\Cal C$, let $\Cal Z=\Cal T\cap\Cal F$, and let $X$ be any object in $\Cal C$. \begin{enumerate}
	\item If $\Hom_{\Cal C}(X,F)=\Triv_{\Cal Z}(X,F)$ for every $F\in\mathcal F$, then $X\in\Cal T $.
	\item If $\Hom_{\Cal C}(T,X)=\Triv_{\Cal Z}(T,X)$ for every $T\in\Cal T $, then $X\in\mathcal F$. 
\end{enumerate}\end{proposition}

\begin{proof} By definition, there is a short $\Cal Z$-preexact sequence $\xymatrix{
	A \ar[r]^f &  X \ar[r]^g &  C}$ with $A\in\Cal T$ and $C\in\Cal F$. If $\Hom_{\Cal C}(X,F)=\Triv_{\Cal Z}(X,F)$ for every $F\in\mathcal F$, then $g$ is $\Cal Z$-trivial, so $f$ is an isomorphism by the previous lemma. As $\Cal T$ is closed under isomorphism, we get that $X\in\Cal T$.  This proves (1). Similarly for (2).\end{proof}
  
  Notice that if $(\Cal T,\Cal F)$ is a pretorsion theory in $\Cal C$, then $(\Cal F,\Cal T)$ is a pretorsion theory in $\Cal C^{\op}$. Clearly, what we have proved in the previous sections shows that $(${\bf Equiv}$,${\bf ParOrd}$)$ is a pretorsion theory in $\pre$.

\begin{remark}{\rm 
There is a notion of torsion theories  for pointed 
categories given by Janelidze and Tholen in \cite{JT}. Their torsion theories 
are a special case of our notion of pretorsion theory. Namely, 
let $\Cal C$ be a pointed category and let $\Cal Z$ be the class of all 
zero objects on $\Cal C$. Then, for any pair of fixed objects $A,B$ of $\Cal C$, the family of all $\Cal Z$-trival morphisms consists  exactly of all zero morphisms. It immediately follows that, if $f\colon A\to B$, $k\colon X\to A$ are morphisms in $\Cal C$, then $k$ is a kernel of $f$ if and only if $k$ is a $\Cal Z$-prekernel of $f$ (where $\Cal Z$ is as before), and similarly for cokernels.  Thus short exact sequences, as defined in  \cite{JT}, are exactly our short $\Cal Z$-preexact sequences. In \cite{JT}, a torsion theory in $\Cal C$ is a pair $(\Cal T,\Cal F)$ of full subcategories of $\Cal C$ that are closed under isomorphism, satisfying the following axioms: 
\begin{enumerate}
	\item $\Hom(T,F)=0$, for every $T\in \Cal T,F\in \Cal F$.
	\item For every object $X$ in $\Cal C$, there are objects $T(X)\in \Cal T, F(X)\in \Cal F$ and a short exact sequence $T(X)\to X\to F(X)$. 
\end{enumerate} 
The fact that every torsion theory $(\Cal T,\Cal F)$ is an example of pretorsion theory will follow if we prove that $\Cal T\cap \Cal F$ is exactly the class $\Cal Z$ of all zero objects of $\Cal C$. Clearly, every zero object is in $\Cal T\cap \Cal F$. Conversely, let $H$ be any object in $\Cal T\cap \Cal F$. By axiom (1), the identity $1_H\colon H\to H$ is the zero morphism, i.e., factors through a zero object. It immediately follows that $H$ is (isomorphic to) a zero object.}\end{remark}

\end{document}